\documentclass[runningheads]{llncs}

\usepackage{url}
\usepackage{multirow}
\usepackage{amssymb}
\usepackage{amsxtra}
\usepackage{makeidx}
\usepackage{amsfonts}
\usepackage{mathtools}
\usepackage{algorithm}
\usepackage{algorithmic}
\usepackage{float}
\usepackage{graphicx}
\usepackage{amsmath}
\usepackage{comment}
\usepackage[font=small,labelfont=bf]{caption}
\usepackage{subcaption}
\usepackage{breqn}
\usepackage[utf8]{inputenc}
\usepackage[T2A]{fontenc}
\usepackage[english]{babel}

\usepackage{color}
\usepackage{hyperref}
\usepackage{appendix}

\DeclareMathOperator*{\argmin}{arg\,min}
\DeclareUnicodeCharacter{202A}{\,}
\DeclareUnicodeCharacter{202A}{ }

\begin{document}
	
\title{Some Adaptive First-order Methods for Variational Inequalities with Relatively Strongly Monotone Operators and Generalized Smoothness\thanks{The work of F. Stonyakin and A. Gasnikov was supported by the strategic academic leadership program <<Priority 2030>> (Agreement  075-02-2021-1316, 30.09.2021).}}


\author{Seydamet S. Ablaev\inst{1, 2} \orcidID{0000-0002-9927-6503} \and 
Alexander A. Titov\inst{2,3} \orcidID{0000-0001-9672-0616} \and 
Fedor S. Stonyakin \inst{1,2} \orcidID{0000-0002-9250-4438} \and
‪Mohammad S. Alkousa\inst{2,3} \orcidID{0000-0001-5470-0182}  \and  
Alexander V. Gasnikov \inst{2,3,4} \orcidID{0000-0002-5982-8983}}
	
\authorrunning{S. Ablaev et al.}
\titlerunning{Adaptive Methods for Variational Inequalities}

\institute{V.\,I.\,Vernadsky Crimean Federal University, Simferopol, Russia
\and 
Moscow Institute of Physics and Technology, Moscow, Russia
\and 
HSE University, Moscow, Russia
\and 
Caucasus Mathematical Center, Adyghe State University, Maikop, Russia\\
\email{seydamet.ablaev@yandex.ru, a.a.titov@phystech.edu, mohammad.alkousa@phystech.edu, fedyor@mail.ru, gasnikov@yandex.ru}
}
\maketitle          
\begin{abstract} In this paper, we introduce some adaptive methods for solving variational inequalities with relatively strongly monotone operators. Firstly, we focus on the modification of the recently proposed, in smooth case \cite{Sidford}, adaptive numerical method for generalized smooth (with H\"older condition) saddle point problem,  which has convergence rate estimates similar to accelerated methods. We provide the motivation for such an approach and obtain theoretical results of the proposed method. Our second focus is the adaptation of widespread recently proposed methods for solving variational inequalities with relatively strongly monotone operators. The key idea in our approach is the refusal of the well-known restart technique, which in some cases causes difficulties in implementing such algorithms for applied problems. Nevertheless, our algorithms show a comparable rate of convergence with respect to algorithms based on the above-mentioned restart technique. Also, we present some numerical experiments, which demonstrate the effectiveness of the proposed methods.  \end{abstract}

\section*{Introduction}\label{sec1_introduction}
Non-smooth convex optimization plays a key role in solving the vast majority of modern applied problems \cite{Bubeck,Cheng,Clarke,Nesterov_2005}. In this paper, we focus on non-smooth saddle point problems and variational inequalities, which, as can be easily shown, are closely related to each other \cite{Stampacchia,Stonyakin_JOTA}. Such settings of the optimization problems naturally arise when considering problems in machine learning \cite{Dauphin,Jin}, data science \cite{Stampacchia}, economic systems \cite{Buiter}, optimal transport \cite{Dafermos}, network equilibrium \cite{Friesz}, game theory \cite{Mertikopoulos_2018}, general equilibrium theory \cite{Cherukuri,Grandmont}, etc.

Remind the problem of solving Minty variational inequality. For a given operator $g: X \rightarrow \mathbb{R}^n$, where $X$ is a closed convex subset of some finite-dimensional vector space, we need to find a vector $x_*\in X$, such that
\begin{equation}\label{VI}
  \langle g(x),x_*-x\rangle\leqslant 0, \quad \forall x\in X.
\end{equation}

The operator $g$ is called $L$-smooth, if for any $x, y, z \in X$ the following inequality holds
\begin{equation}\label{rel_smooth}
    \langle g(y)-g(z),x-z\rangle \leqslant LV(x,z) + LV(z,y),
\end{equation}
where $V(\cdot, \cdot)$ is the distance in some generalized sense, namely, Bregman divergence (see \eqref{Bregmann_def}, below).
We also assume, that the operator $g$ is $\mu$--relatively strongly monotone, i.e.
	\begin{equation}\label{eq:rel_str_mon}
	 	\mu V(y, x) + \mu V(x, y) \leqslant \langle g(y) - g(x), y - x \rangle, \quad  {\color{black} \forall  x,y \in X.}
	 \end{equation}

The concept of relative strong monotonicity is a natural generalization of the concept of relative strong convexity of the objective functional \cite{Lu_nest} {\color{black} in optimization problems,} for variational inequalities.

We are motivated by the following saddle point problem
\begin{equation}\label{main1}
\min\limits_{x}\max\limits_{y} f(x,y).
\end{equation} 

Using the recently proposed new paradigm of convex optimization \cite{Bauschke,Nesterov_2019}, namely, relative smoothness condition, there was proposed a technique \cite{Sidford}, which provides the possibility of the acceleration of numerical methods for solving saddle point problems, assuming, that the gradient of the objective function $\nabla f(x,y)$ satisfies Lipschitz condition. Moreover, the proposed method is adaptive with respect to all Lipschitz constants of the objective's gradient.

In this paper, we extend the considered class of saddle point problems and replace the classical Lipschitz continuity condition
\begin{equation}
\|\nabla f(z)-\nabla f(u)\| \leqslant L\|z-u\|,
\end{equation}
by the following H\"older continuity condition
\begin{equation}\label{holder}
\|\nabla f(z)-\nabla f(u)\| \leqslant L_{\nu}\|z-u\|^{\nu},
\end{equation}
{\color{black} $z = (x_z, y_z), u = (x_u, y_u)$,} with respect to the gradient of the objective, 
where $0\leqslant\nu\leqslant1$. Hereinafter we consider arbitrary non-Euclidean norms{\color{black}, which are} defined on the corresponding spaces unless otherwise stated.

Note, that the concept of H\"older continuity is an extremely important generalization of the Lipschitz continuity condition. A huge number of applied problems can be formulated exclusively on the class of minimization of H\"older continuous functionals, e.g. smooth multi-armed bandit problem \cite{Liu}, detecting heart rate variability \cite{Nakamura}, etc.
Also, if some function is uniformly convex, then its conjugated will necessarily have the H\"older-continuous gradient, according to \cite{Nesterov_2015}.

Based  on the recently proposed restart technique of the Universal Proximal Method for solving variational inequalities \cite{Gasnikov}, we propose algorithms, which ensure the $\varepsilon$--approximate solution of the considered problem \eqref{VI} after no more than
\begin{equation}
N = \left \lceil \frac{2L\Omega}{\mu}\log_2\frac{R_0^2}{\varepsilon} \right \rceil
\end{equation}
iterations, where $\mu$ denotes the constant of relative strong monotonicity of $g$, $\Omega$ and $R_0$ are specified in Algorithm \ref{Alg:RUMP} and can be understood as  some characteristics of the domain of the {\color{black} operator $g$.} 

The paper consists of an introduction and four main sections. In Sect. \ref{section_Saddle_Point_Problem}, we discuss approach \cite{Sid,Sidford} to the accelerated rates of first-order methods for strongly convex-concave saddle point problems basing on relative smoothness and strong monotonicity conditions. Moreover, we consider generalizations of smoothness conditions for saddle point problems. In Sect. \ref{RestartUMP}, we propose adaptive version of restarted Mirror Prox method \cite{Inex} for generalized smooth problems. Further, in Sect. \ref{sect:adapt_wth}, we propose some adaptive methods, which do not imply the restart technique, but have the similar convergence rate estimates of the proposed algorithm with restart technique. {\color{black}In Sect. \ref{sect_numerical}, we present some numerical experiments for the saddle point problem and Minty variational inequality, which demonstrate the effectiveness of the proposed methods.}

The contributions of the paper can be formulated as follows.
\begin{itemize}
	\item We consider the non-smooth strongly convex-concave saddle point problem and propose a restarted version of the Universal Proximal Method for the corresponding variational inequality, which {\color{black}guarantees the} $\varepsilon$--approximate solution of the problem \eqref{VI} in an optimal rate of $N = \left \lceil \frac{2L\Omega}{\mu}\log_2\frac{R_0^2}{\varepsilon}\right \rceil$ iterations. These algorithms are of interest in case of a huge value of the condition number $\frac{L}{\mu}$ as well as in the case of considering not strongly convex-concave saddle point problems.
	
	\item We propose methods beyond the restart technique and show, that in some cases they may have even better estimates of the convergence rate compared to methods, based on the restart technique. Moreover, the required number of iterations of such algorithms does not exceed $N = \left \lceil \frac{L+\mu}{\mu}\log_2\frac{R_0^2}{\varepsilon}\right \rceil$.
	
	\item We present some numerical experiments, which demonstrate the effectiveness of the proposed methods.

\end{itemize}

We start with some auxiliaries. Let $E$ be a finite-dimensional vector space and $E^*$ be its dual. Let us choose some norm $\| \cdot \|$ on $E$. Define the dual norm $\|\cdot \|_*$ as follows
\begin{equation}
    \|\phi\|_{*} = \max\limits_{\|x\|\leqslant 1} \{ \langle \phi, x \rangle\},
\end{equation}
where $\langle \phi, x \rangle$ denotes the value of the linear function $\phi \in E^*$ at the point $x\in E$.

Let us choose some prox-function $d(x)$, which is continuously differentiable and convex on $E$, and define the corresponding Bregman divergence as follows
\begin{equation}\label{Bregmann_def}
    V(y,x) = V_d(y,x) = d(y) - d(x) - \langle \nabla d(x), y - x \rangle, \quad \forall x, y \in E.
\end{equation}
The Bregman divergence can be understood as some generalization of the distance in the considered set.

\section{Towards Adaptive Accelerated Rates for Saddle Point Problems with Generalized Smoothness Condition} \label{section_Saddle_Point_Problem}

Let $Q_x \subset \mathbb{R}^n$ and $Q_y \subset \mathbb{R}^m$ be nonempty, convex and compact sets.  
Consider the following saddle point problem

\begin{equation}\label{main}
\min\limits_{x\in Q_x}\max\limits_{y\in Q_y} \{f(x,y)+h(x)-g(x)\},
\end{equation} 
where  $f(x,y): Q_x\times Q_y\rightarrow\mathbb{R}$ is convex function for fixed $y \in Q_y$ and concave for fixed $x\in Q_x$, functions $h(x)$ and $g(y)$ are convex on $Q_x$ and $Q_y$,  respectively, and for each $x{\color{black}, x'} \in Q_x, y{\color{black}, y'} \in Q_y${\color{black}, we have}
$$
\| \nabla h(x) - \nabla h(x')\|_* \leqslant L_{x}\|x-x'\|, \quad 
\| \nabla g(y) - \nabla g(y')\|_* \leqslant L_{y}\|y-y'\|,
$$
{\color{black} for some $L_x > 0, L_y > 0$.}

\begin{remark}
If $f(x,y): Q_x\times Q_y\rightarrow\mathbb{R}$ is strongly convex function for fixed $y \in Q_y$ and strongly concave for fixed $x\in Q_x$, we can consider the problem \eqref{main} with $h(x) = \frac{\|x\|^2}{2}$ and  $g(y) = \frac{\|y\|^2}{2}$.
\end{remark}

Let us consider the following setting of the problem \eqref{main}. Suppose, for any $x,x' \in Q_x,\  y, y'\in Q_y$, $L_{xx} > 0, L_{yy} >0, L_{xy} > 0$, and for some $\nu \in [0,1]$, the following inequalities hold

\begin{equation} \label{Lxx_Lxy}
\| \nabla_xf(x,y) - \nabla_xf(x',y')\|_* \leqslant L_{xx}\|x-x'\|^{\nu} + L_{xy}\|y-y'\|^{\nu},
\end{equation} 
\begin{equation} \label{Lxy_Lyy}
\| \nabla_yf(x,y) - \nabla_yf(x',y')\|_* \leqslant L_{xy}\|x-x'\|^{\nu} + L_{yy}\|y-y'\|^{\nu}.
\end{equation} 

Let $\omega_{x}, \omega_{y}$ be some 1-strongly convex function w.r.t. $\|\cdot\|_{Q_x},\|\cdot\|_{Q_y}$ respectively, $\alpha: Q_{x}^{*} \rightarrow \mathbb{R}$ be the convex conjugate of $\frac{\|x\|^2}{2}-\mu_{x} \omega_{x}$, and $\beta: {Q_y}^{*} \rightarrow \mathbb{R}$ be the convex conjugate of $\frac{\|y\|^2}{2}-\mu_{y} \omega_{y}$. Then we can consider the following saddle point problem
\begin{equation}\label{mainmod}
\min\limits_{\substack{x \in Q_x,\\ b \in Q_y^*}} \max\limits_{\substack{y \in Q_y,\\ a \in Q_x^*}}\Big\{\langle a, x\rangle+\langle b, y\rangle+\mu_{x} \omega_{x}(x)-\mu_{y} \omega_{y}(y)+f(x, y)-\alpha(a)+\beta(b)\Big\}.
\end{equation}

It is shown \cite{Sidford}, that if $\left(x^{*}, y^{*}, a^{*}, b^{*}\right)$ is the saddle point to \eqref{mainmod}, then $\left(x^{*}, y^{*}\right)$ is the saddle point to \eqref{main}. Also,
$\alpha$ is $\frac{1}{L_{x}}$-strongly convex in $\|\cdot\|_{{Q_x}, *}$, and $\beta$ is\\ $\frac{1}{L_{y}}$-strongly convex in $\|\cdot\|_{{Q_y},*}$ \cite{Sidford}.

\begin{lemma}\label{lemma_1}
Define the following operator ($(x, y) \in Q_x \times Q_y := X$)
\begin{equation*}
\begin{aligned}
g(x,y,a,b)=\Big( a+\mu_{x} \nabla \omega_{x}(x)+ &\nabla_{x} f(x, y),-b+\mu_{y} \nabla \omega_{y}(y)-\nabla_{y} f(x, y), \\& \qquad \qquad \qquad \qquad \qquad -x+\nabla \alpha(a), y+\nabla \beta(b) \Big),
\end{aligned}
\end{equation*}

where $f$ satisfies \eqref{Lxx_Lxy}-\eqref{Lxy_Lyy}. Consider the following prox-function 
$$
    d(x,y,a,b)=\mu_{x} \omega_{x}(x)+\mu_{y} \omega_{y}(y)+\alpha(a)+\beta(b),
$$ 
and the corresponding Bregman divergence, defined according to \eqref{Bregmann_def}. Then $g$ is $1$-relatively strongly monotone, i.e. 
    \begin{equation}\label{eq2}
	 	 \langle g(y) - g(x), y - x \rangle \geqslant  V(y, x) +  V(x, y),
    \end{equation}
and generalized relatively smooth operator, i.e.
    \begin{equation}\label{eq3}
    \langle g(y)-g(z),y-x\rangle \leqslant LV(y,z) + LV(x,y) +\delta,
    \end{equation}
for some $\delta > 0$, with  
\begin{equation}\label{Lconst}
L=\widetilde{L}(\delta) = \left(\frac{2}{\delta}\right)^{\frac{1-\nu}{1+\nu}}\left(\frac{L_{xx}^{\frac{2}{1+\nu}}}{\mu_x} + \frac{L_{xy}^{\frac{2}{1+\nu}}}{\sqrt{\mu_x\mu_y}}+ \frac{L_{yy}^{\frac{2}{1+\nu}}}{\mu_y}\right).
\end{equation}
\end{lemma}


\begin{proof}
The proof is given in arXiv preprint \cite{arxiv_version}.
\end{proof}

Hence, considered variational inequalities with relatively strongly monotone and generalized relatively smooth operators {\color{black}allow} one to obtain first-order method complexity estimates for the corresponding class of strongly convex-convex saddle point problems, which are similar to the accelerated methods  \cite{Sidford}. Moreover, using the artificial inaccuracy, Lemma \ref{lemma_1} extends this approach to saddle {\color{black}point} problems with generalized smoothness conditions \cite{Stonyakin_JOTA,Inex}.

However, extending the class of problems, one can potentially encounter the problem of a large value of $\widetilde{L}(\delta)$. On the other hand, even while considering the smooth case for saddle point problems, it may be difficult to estimate all the 5 parameters {\color{black} $\mu_x$, $\mu_y$}, $L_{xx}$, $L_{xy}$ and $L_{yy}$. Motivated by this and starting from the methodology of Y.~E.~Nesterov's works \cite{Devolder,Nesterov_2013,Nesterov_2015}, we propose methods allowing the adaptively selection of the corresponding values of these parameters.

\section{Adaptive Restarted Mirror Prox for Variational Inequalities with Relative Strongly Monotone Operators}
\label{RestartUMP}

Recently \cite{Inex}, there was proposed an adaptive universal algorithm (listed as Algorithm \ref{Alg:UMP}, below), which can automatically adjust to the smoothness level of the operator $g$.

\begin{algorithm}[htp]
\caption{Universal Mirror Prox for Variational Inequalities \cite{Inex}.}
\label{Alg:UMP}
\begin{algorithmic}[1]
   \REQUIRE $\varepsilon > 0$, $\delta >0$, $x_0 \in X$, initial guess $L_0 >0$, prox-setup: $d(x)$, $V(x,z)$.
   \STATE Set $k=0$, $z_0 = \arg \min_{u \in X} d(u)$.
   \REPEAT
			\STATE Find the smaller $i_k\geqslant  0$, such that
			\begin{equation*}
			\begin{aligned}
			    \langle g(z_k),z_{k+1}-z_k\rangle & \leqslant \langle g(w_k),z_{k+1}-w_k\rangle + \langle g(z_k),w_{k}-z_k\rangle +  \\ &\qquad \qquad \qquad   +L_{k+1}(V(w_k,z_k) + V(z_{k+1},w_k)) +\delta,
			    \end{aligned}
			\end{equation*}
			
			where $L_{k+1} = 2^{i_k-1}L_k$, and 
    			\begin{equation*}
    			     w_k = \arg\min\limits_{x \in X}\{\langle g(z_k),x-z_k\rangle + L_{k+1}V(x,z_k)\},
    			\end{equation*}
    				\begin{equation*}
    			     z_{k+1} = \arg\min\limits_{x \in X}\{\langle g(w_k),x-w_k\rangle + L_{k+1}V(x,z_k)\}.
    			\end{equation*}
			\UNTIL{
		        $S_N := \sum\limits_{k=0}^{N-1}\frac{1}{L_{k+1}}\geqslant  \frac{\max_{x\in X}V(x,x_0)}{\varepsilon}.$
			}
		\ENSURE $z_N$.
\end{algorithmic}
\end{algorithm} 

\begin{theorem}[\cite{Inex}]
Let $g$ be a monotone operator, {\color{black}$z_N$} be  the output of Algorithm \ref{Alg:UMP}{\color{black} after $N$ iterations}. Then the following inequality holds
\begin{equation}
\langle g(x_*),{\color{black}z_N} - x_*\rangle \leqslant -\frac{1}{S_N}\sum\limits_{k=0}^{N-1}\frac{\langle g(w_k),x_*-w_k \rangle}{L_{k+1}} \leqslant \frac{2L V(x_*,z_0)}{N}.
\end{equation}
Moreover, the total number of iterations does not exceed
\begin{equation}
    N =  \left \lceil \frac{2L}{\varepsilon} \cdot \max\limits_{x\in X}V(x_0,x) \right \rceil. 
\end{equation}
\end{theorem}

\begin{lemma}\label{lemma2}
Let $g$ be a relatively strongly monotone operator. For Algorithm \ref{Alg:UMP}, the following $\delta$-decreasing of Bregman divergence takes place
\begin{equation}\label{estimate_UMP}
V(x_*,z_{N})\leqslant V(x_*,z_{0}) +\delta S_N.
\end{equation}
\end{lemma}
\begin{proof}
The proof is given in arXiv preprint \cite{arxiv_version}.
\end{proof}

The following Algorithm \ref{Alg:RUMP} provides the possibility of the acceleration of the proposed Algorithm \ref{Alg:UMP} for solving variational inequality with relatively strongly monotone operator.

\begin{algorithm}[htp]
	\caption{Restarted version of Algorithm \ref{Alg:UMP}.}
	\label{Alg:RUMP}
	\begin{algorithmic}[1]
		\REQUIRE $\varepsilon > 0$, $\mu >0$, $\Omega$ : $d(x) \leqslant \frac{\Omega}{2} \ \forall x\in X: \|x\| \leqslant 1$; $x_0,\; R_0 \ : V(x_*,x_0) \leqslant R_0^2.$
		\STATE $p=0,d_0(x)=R_0^2d\left(\frac{x-x_0}{R_0}\right)$.
		\REPEAT
		\STATE $x_{p+1}$ --- output of Algorithm \ref{Alg:UMP} with prox function $d_{p}(\cdot)$ and stopping criterion $S_N :=\sum_{i=0}^{N-1}L_{i+1}^{-1}\geqslant  \frac{\Omega}{\mu}$.
		\STATE $R_{p+1}^2 = \frac{\Omega R_0^2}{2^{(p+1)}\mu S_{N_p}}$.
		\STATE $d_{p+1}(x) \leftarrow R_{p+1}^2d\left(\frac{x-x_{p+1}}{R_{p+1}}\right)$.
		\STATE $p=p+1$.
		\UNTIL			
		$p > \log_2\left(\frac{2R_0^2}{\varepsilon}\right).$	
		\ENSURE $x_p$.
	\end{algorithmic}
\end{algorithm}

\begin{theorem}\label{th2}
Let $g$ be a {\color{black}generalized relatively smooth \eqref{eq3}} and $\mu$-relatively strongly monotone operator. Then for the output point $x_p$ of the Algorithm \ref{Alg:RUMP}, it will be hold:  
$V(x_*,x_p)\leqslant \varepsilon+ \frac{2 \Omega L\delta}{\mu^2}$. Moreover, the total number of iterations of Algorithm \ref{Alg:RUMP} does not exceed 
\begin{equation}\label{num_iter_alg2}
N = \left\lceil\frac{2L\Omega}{\mu} \cdot \log_2\frac{R_0^2}{\varepsilon}\right\rceil.
\end{equation}
\end{theorem}
\begin{proof}
The proof is given in arXiv preprint \cite{arxiv_version}.
\end{proof}

\begin{remark}
As shown above, Algorithm \ref{Alg:RUMP} needs no more than $N = \left\lceil\frac{2L\Omega}{\mu}\cdot \log_2\frac{R_0^2}{\varepsilon}\right \rceil$ iterations to provide a solution of the problem \eqref{main}, $L$ is defined according to \eqref{Lconst}, while the technique, described in \cite{Main} has the following worse complexity estimate
\begin{equation}
    \inf_{\nu \in [0,1]} \left\lceil  \left(\frac{L_{\nu}}{\mu}\right)^{\frac{2}{1+\nu}} \cdot \frac{2^{\frac{2}{1+\nu}}\Omega }{\varepsilon^{\frac{1-\nu}{1+\nu}}} \cdot \log_2 \frac{2 R_0^2}{\varepsilon}\right\rceil,
\end{equation}
where $L_{\nu} = \tilde{L}\left(\frac{\tilde{L}}{2\varepsilon}\frac{(1-\nu)(2-\nu)}{2-\nu} \right)^{\frac{(1-\nu)(1+\nu)}{2-\nu}},
\tilde{L} = \left( L_{xy}\left(\frac{2L_{xy}}{\mu_y}\right)^{\frac{\nu}{2-\nu}}+L_{xx}D^{\frac{\nu-\nu^2}{2-\nu}} \right)$, and $D$ is the diameter of the domain of $f(x,\cdot)$.
\end{remark}

\begin{remark}
The obtained estimate 
\begin{equation}
N = \left\lceil 2\Omega\left(\frac{L_{xx}}{\mu_x} + \frac{L_{xy}}{\sqrt{\mu_x\mu_y}}+ \frac{L_{yy}}{\mu_y}\right)\cdot  \log_2\frac{R_0^2}{\varepsilon}\right\rceil
\end{equation}
is optimal for saddle point problems \eqref{mainmod} with $\nu = 1$.

\end{remark}

\begin{remark}
Note, that $\Omega$ may depend on the dimension of the considered space \cite{Nemir_lect}.
\end{remark}

\section{First-order methods for relatively strongly monotone variational inequalities beyond the restart technique}\label{sect:adapt_wth}

Basing on some recently proposed methods \cite{Sid,Sidford} for VIs with 
strongly relatively monotone operators $g :X \rightarrow\mathbb{R}^n$, we consider  algorithms (see Algorithms \ref{alg2}, \ref{alg1} and \ref{alg3}) without using the restart technique. Similarly to the previous section we consider the case of operators $g$ with the generalized smoothness condition \eqref{eq3}. We improve the quality of the solution, compared to Algorithm \ref{alg2} by reducing $O\left(\frac{\delta}{\mu^2}\right)$ to $O\left(\frac{\delta}{\mu}\right)$, which provides the better convergence rate in case of small value of $\mu$. It is also worth noting, that proposed algorithms do not require knowledge of the parameter $\Omega$.

\begin{algorithm}[h!]
\caption{Adaptive first-order method for variational inequalities with $\mu$-relatively strongly monotone operators without restarts.}
\label{alg2}
\begin{algorithmic}[1]
\REQUIRE $\varepsilon>0, \delta >0,  x_0 \in X, L_0 > 0, \mu >0, d(x), V(x, z).$
\STATE Set $z_0 = \argmin_{u\in X} d(u).$
\FOR{$k \geqslant  0 $}
\STATE Find the smallest integer $i_k \geqslant  0$, such that
\begin{equation}\label{eq_alg_2}
\langle g(z_k) - g(w_k), z_{k+1} - w_k \rangle \leqslant L_{k+1} \left (V(w_k, z_k) + V(z_{k+1}, w_k) \right ) + \delta,
\end{equation}
 where $L_{k+1} = 2^{i_k - 1} L_k$, and
\begin{equation}\label{wk_alg2}
    w_k = \argmin\limits_{y\in X} \left \{ \left \langle \frac{1}{L_{k+1}} g(z_k), y \right \rangle  + V(y, z_k)\right \},
\end{equation}
\begin{equation}\label{zk1_alg2}
    z_{k+1} = \argmin\limits_{z \in X}\left \{ \left \langle \frac{1}{L_{k+1}} g(w_k), z \right \rangle + V(z, z_k) + \frac{\mu}{L_{k+1}} V(z, w_k) \right \}.
\end{equation}
\ENDFOR
\ENSURE $z_k$.
\end{algorithmic}
\end{algorithm}

\begin{theorem}\label{th4}
Let $g$ be a $\mu$-relatively strongly monotone operator, and $z_*$ be the exact solution of the variational inequality \eqref{VI}. Then for Algorithm \ref{alg2}, the following inequality holds 
\begin{equation}\label{estim_alg4}
\begin{aligned}
V(z_*, z_{k+1}) \leqslant  \prod_{i = 1}^{k+1} \left ( 1 + \frac{\mu}{L_{i}} \right )^{-1} & V(z_*, z_0)  + \frac{\delta}{L_{k+1} + \mu} + 
\\& \quad \quad   + \sum_{j = 1}^k \frac{\delta}{L_j + \mu} \prod_{i = j+1}^{k+1} \left ( 1 + \frac{\mu}{L_i} \right )^{-1}.
\end{aligned}
\end{equation}
\end{theorem}
\begin{proof}
The proof is given in arXiv preprint \cite{arxiv_version}.
\end{proof}


\begin{algorithm}[htp]
\caption{Adaptive first-order method for variational inequalities with $\mu$-relatively strongly monotone and $L$-smooth operators.}
\label{alg1}
\begin{algorithmic}[1]
\REQUIRE $\varepsilon>0, x_0 \in X, L_0 > 0, \mu >0, d(x), V(x, z).$
\STATE Set $ z_0 = \argmin_{u\in X} d(u).$
 \FOR{$k\geqslant 0$}
\STATE Find smallest integer $i_k \geqslant  0$, such that
\begin{equation}\label{eq_alg}
\langle g(z_k) - g(w_k), z_{k+1} - w_k \rangle \leqslant L_{k+1} \left (V(w_k, z_k) + V(z_{k+1}, w_k) \right ),
\end{equation}
 where $L_{k+1} = 2^{i_k - 1} L_k$, and
\begin{equation}\label{wk_alg1}
    w_k = \argmin_{y \in X} \left \{ \left \langle \frac{1}{L_{k+1}} g(z_k), y \right \rangle  + V(y, z_k)\right \},
\end{equation}
\begin{equation}\label{zk1_alg1}
    z_{k+1} = \argmin_{z \in X} \left \{ \left \langle \frac{1}{L_{k+1}} g(w_k), z \right \rangle + V(z, z_k) + \frac{\mu}{L_{k+1}} V(z, w_k) \right \}.
\end{equation}
\ENDFOR
\ENSURE $z_k$.
\end{algorithmic}
\end{algorithm}

\begin{corollary}\label{th3}
Let $g$ be a $\mu$-relatively strongly monotone operator, and $z_*$ be the exact solution of the variational inequality \eqref{VI}. Then for Algorithm \ref{alg1}, the following inequalities hold
\begin{equation}\label{estim_alg3}
V(z_*, z_{k+1}) \leqslant \prod_{i=0}^k\left ( 1 + \frac{\mu}{L_{i+1}} \right )^{-1}  V(z_*, z_0),
\end{equation}
\begin{equation}\label{eq1}
V(z_*, z_{k+1}) \leqslant \left ( 1 + \frac{\mu}{2 L} \right )^{-(k+1)} V(z_*, z_0).
\end{equation}
\end{corollary}
\begin{remark}
Due to \eqref{eq1}, the number of iterations of Algorithm \ref{alg1} for solving the problem \eqref{VI} does not exceed $N = \left \lceil \frac{2L+\mu}{\mu}\log_2\frac{R_0^2}{\varepsilon}\right \rceil$, which coincides with \eqref{num_iter_alg2} up to the multiplication by a constant in the case of $L\sim \mu\sim 1.$
\end{remark}

\begin{remark}
Taking into account \eqref{Lconst}, we find that the inequality \eqref{eq1}  has the following form
\begin{equation*}
V(z_*, z_{k+1})  
= \left ( \frac{2\widetilde{L}}{2\widetilde{L} + \mu} \right )^{k+1} V(z_*, z_0),
\end{equation*}

where $\widetilde{L}$ is given in \eqref{Lconst}.

\end{remark}

The main difference between Algorithms \ref{alg2} and \ref{alg1} and the next Algorithm \ref{alg3} is a modified exit criterion, which leads to decreasing of the coefficient at $\delta$.

\begin{algorithm}[htp]
\caption{The third adaptive first-order method for variational inequalities with $\mu$-relatively strongly monotone operators.}
\label{alg3}
\begin{algorithmic}[1]
\REQUIRE $\varepsilon>0, \delta >0, x_0 \in X, L_0 > 0, \mu >0, d(x), V(x, z).$
\STATE Set $z_0 = \argmin_{u\in X} d(u).$
\FOR{$k \geqslant  0$}
\STATE Find the smallest integer  $i_k \geqslant  0$, such that
\begin{equation}\label{eq_alg_3}
\langle g(z_k) - g(w_k), z_{k+1} - w_k \rangle \leqslant L_{k+1} \left (V(w_k, z_k) + V(z_{k+1}, w_k) \right ) + L_{k+1} \delta,
\end{equation}
where $L_{k+1} = 2^{i_k - 1} L_k$, and
\begin{equation}\label{i1alg3}
    w_k = \argmin\limits_{y \in X} \left \{ \left \langle \frac{1}{L_{k+1}} g(z_k), y \right \rangle  + V(y, z_k)\right \},
\end{equation}
\begin{equation}\label{i2alg3}
    z_{k+1} = \argmin\limits_{z \in X} \left \{ \left \langle \frac{1}{L_{k+1}} g(w_k), z \right \rangle + V(z, z_k) + \frac{\mu}{L_{k+1}} V(z, w_k) \right \}.
\end{equation}
\ENDFOR
\ENSURE $z_k$.
\end{algorithmic}
\end{algorithm}

\begin{theorem}\label{th5}
Let $g$ be a $\mu$-relatively strongly monotone operator, and $z_*$ be the exact solution of the variational inequality \eqref{VI}. Then for Algorithm \ref{alg3}, we have 
\begin{equation}\label{estim_alg5}
V(z_*, z_{k+1}) \leqslant \prod_{i=1}^{k+1} \left ( 1 + \frac{\mu}{L_{i}} \right )^{-1} V(z_*, z_0) + \delta \left (1 + \sum_{j=1}^k \prod_{i=j+1}^{k+1} \left ( 1 + \frac{\mu}{L_i} \right )^{-1} \right ),
\end{equation} 
and
\begin{equation}\label{ineq1}
V(z_*, z_{k+1}) \leqslant \left ( 1 + \frac{\mu}{2L} \right )^{-(k+1)} V(z_*, z_0) + \delta \left ( 1 + \frac{2L}{\mu} \right ).
\end{equation}
\end{theorem}
\begin{proof}
The proof is given in arXiv preprint \cite{arxiv_version}.
\end{proof}

\begin{remark}
If we iterate the inequality
$$
    V(z_*, z_{k+1}) \leqslant \left ( 1 + \frac{\mu}{L_{k+1}} \right )^{-1} V(z_*, z_{k}) + \frac{L_{k+1 }\delta}{L_{k+1} + \mu},
$$
then for the Algorithm \ref{alg3}, the following inequality holds
$$
    V(z_*, z_{k+1}) \leqslant \prod_{i=1}^{k+1} \left ( 1 + \frac{\mu}{L_{i}} \right )^{-1} V(z_*, z_0) + \frac{L_{k+1} \delta}{L_{k+1} + \mu} + \sum_{j=1}^k \frac{L_j \delta}{L_j + \mu} \prod_{i = j+1}^{k+1} \left ( 1 + \frac{\mu}{L_i}\right )^{-1}.
$$
\end{remark}

\begin{remark}
For the selected type of saddle point problems \eqref{mainmod} with \eqref{Lconst} and $\mu = 1$, we find that the inequality 
\begin{equation*}
V(z_*, z_{k+1}) \leqslant \left ( 1 + \frac{\mu}{2L} \right )^{-(k+1)} V(z_*, z_0) + \delta \left ( 1 + \frac{2L}{\mu} \right ),
\end{equation*}
has the following form

\begin{equation*}
    V(z_*, z_{k+1}) \leqslant \left ( 1 + \frac{1}{2\widetilde{L}} \right )^{-(k+1)} V(z_*, z_0) + \widetilde{L}2^{\frac{2}{1-\nu}} \delta^{\frac{2 \nu(1-\nu)}{(1+\nu)^2}}  + \delta,
\end{equation*}
where $\widetilde{L}$ is given in \eqref{Lconst}.
\end{remark}

\section{Numerical Experiments}\label{sect_numerical}



\subsection{Saddle point problem for the smallest covering ball problem with functional constraints}

In this subsection, we consider an example of the Lagrange saddle point problem induced by a problem with geometrical nature, namely, an analogue of the well-known {\it smallest covering ball problem} with functional constraints. This example is equivalent to the following non-smooth convex optimization problem with functional constraints 
\begin{equation}\label{problem_smallest}
	\min_{x \in X} \left\{\psi(x):= \max_{1\leqslant k \leqslant s} \|x - A_k\|_2^2; \; \varphi_p(x) \leqslant 0, \; p=1,...,m \right\},
\end{equation}
where $A_k \in \mathbb{R}^n, k=1,...,s$ are given points and $X$ is a convex compact set. Functional constraints $\varphi_p$,  for $p=1,...,m$, have the following form
\begin{equation}\label{linear_constraints}
	\varphi_p(x):= \sum_{i=1}^n \alpha_{pi}x_i^2 -5, \quad p = 1, ..., m.
\end{equation}



See \cite{arxiv_version}, for more details about the setting of this problem and the setting of its connected conducted experiments with it. The coefficients $\alpha_{pi}$ in \eqref{linear_constraints} are drawn randomly from the following distributions.

\noindent
 \textbf{Case 1:} Pareto II or Lomax distribution with shape equalling $10$.
    
    \noindent
 \textbf{Case 2:} chi-square distribution, with a number of degrees of freedom, equals 3.



For case 1, the results of the work of Algorithms \ref{alg2}, \ref{alg1},  and \ref{alg3}  are presented in the Fig. \ref{fig1}, below.  These results demonstrate the theoretical estimate \eqref{estim_alg4},  \eqref{estim_alg3}  and \eqref{estim_alg5} for Algorithms \ref{alg2}, \ref{alg1}  and \ref{alg3}, respectively. Also, they demonstrate the value of the objective function in \eqref{problem_smallest} at the output point $x_N$ of the algorithm after performing $N$ iterations.
 
For case 2, we compare the work of Algorithms \ref{Alg:RUMP}, \ref{alg2}, \ref{alg1},  and \ref{alg3}.  
The results are presented in Fig. \ref{fig2}. They  demonstrate the theoretical estimate \eqref{estim_alg4}, \eqref{estim_alg3} and \eqref{estim_alg5} for Algorithms \ref{alg2}, \ref{alg1} and \ref{alg3}, respectively. Also, in the comparison with Algorithm \ref{Alg:RUMP}, these results demonstrate the value of the objective function at the output point $x_N$ of the algorithms after performing $N$ iterations. 

\begin{figure*}[t]
\centering
	\begin{minipage}{0.49\textwidth}
		\centering
		\includegraphics[width=1\textwidth]{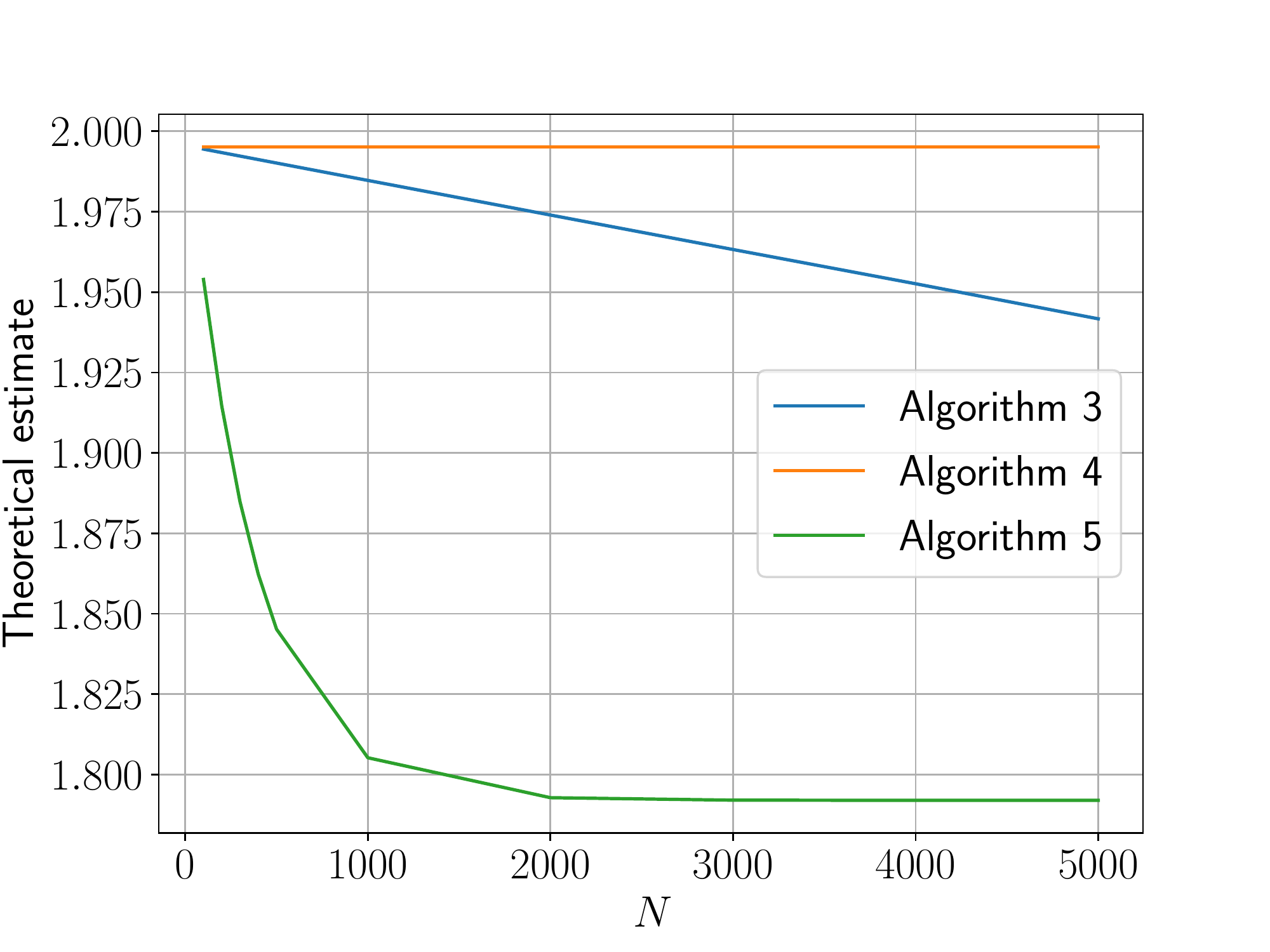}
	\end{minipage}
	\begin{minipage}{0.49\textwidth}
		\centering
		\includegraphics[width=1\textwidth]{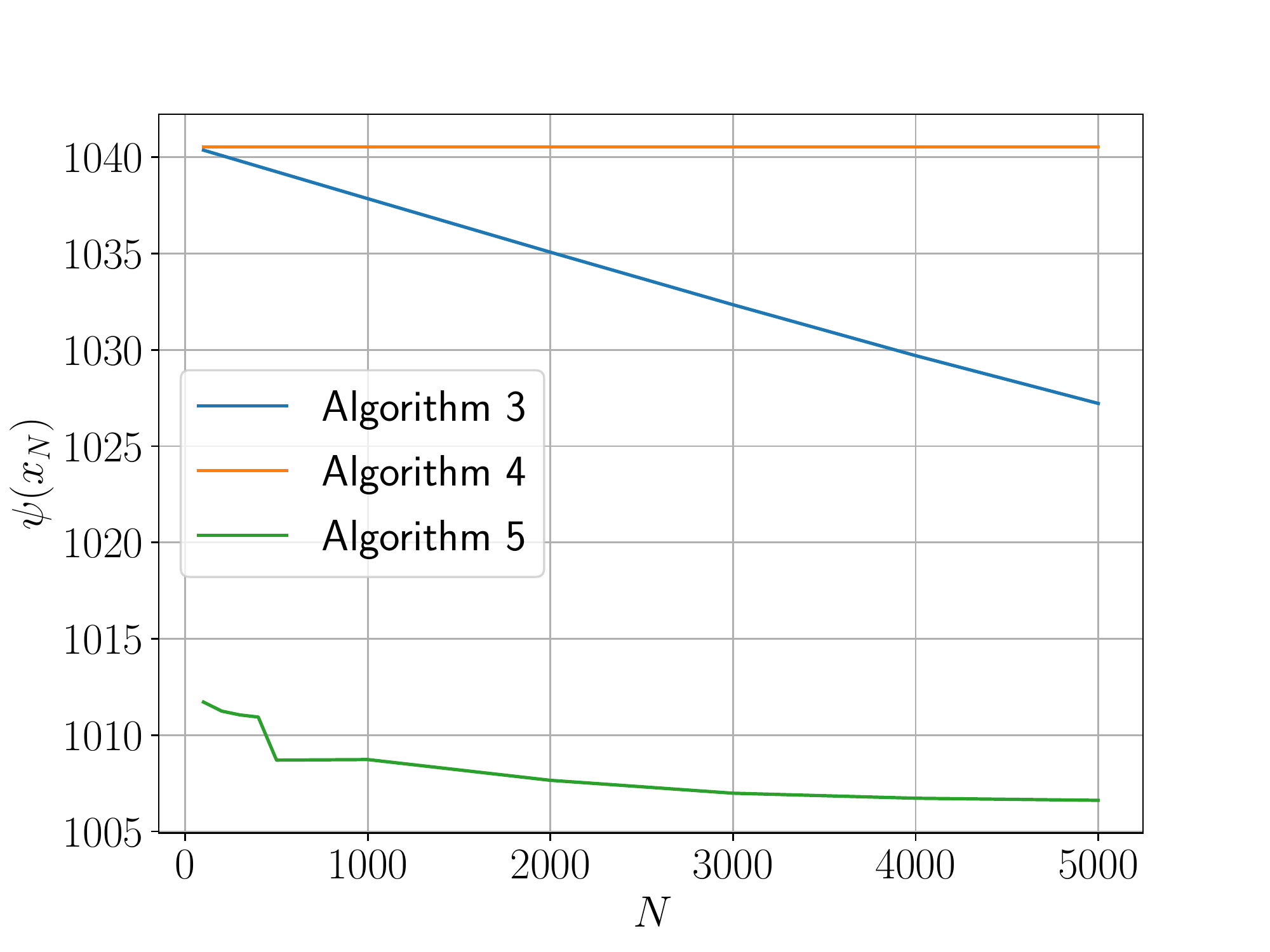}
	\end{minipage}
	\caption{\footnotesize  The results of Algorithms \ref{alg2}, \ref{alg1}   and \ref{alg3}, for case 1. }
	\label{fig1}
\end{figure*}

\begin{figure*}[t]
	\begin{minipage}{0.32\textwidth}
		\centering
		\includegraphics[width=1\textwidth]{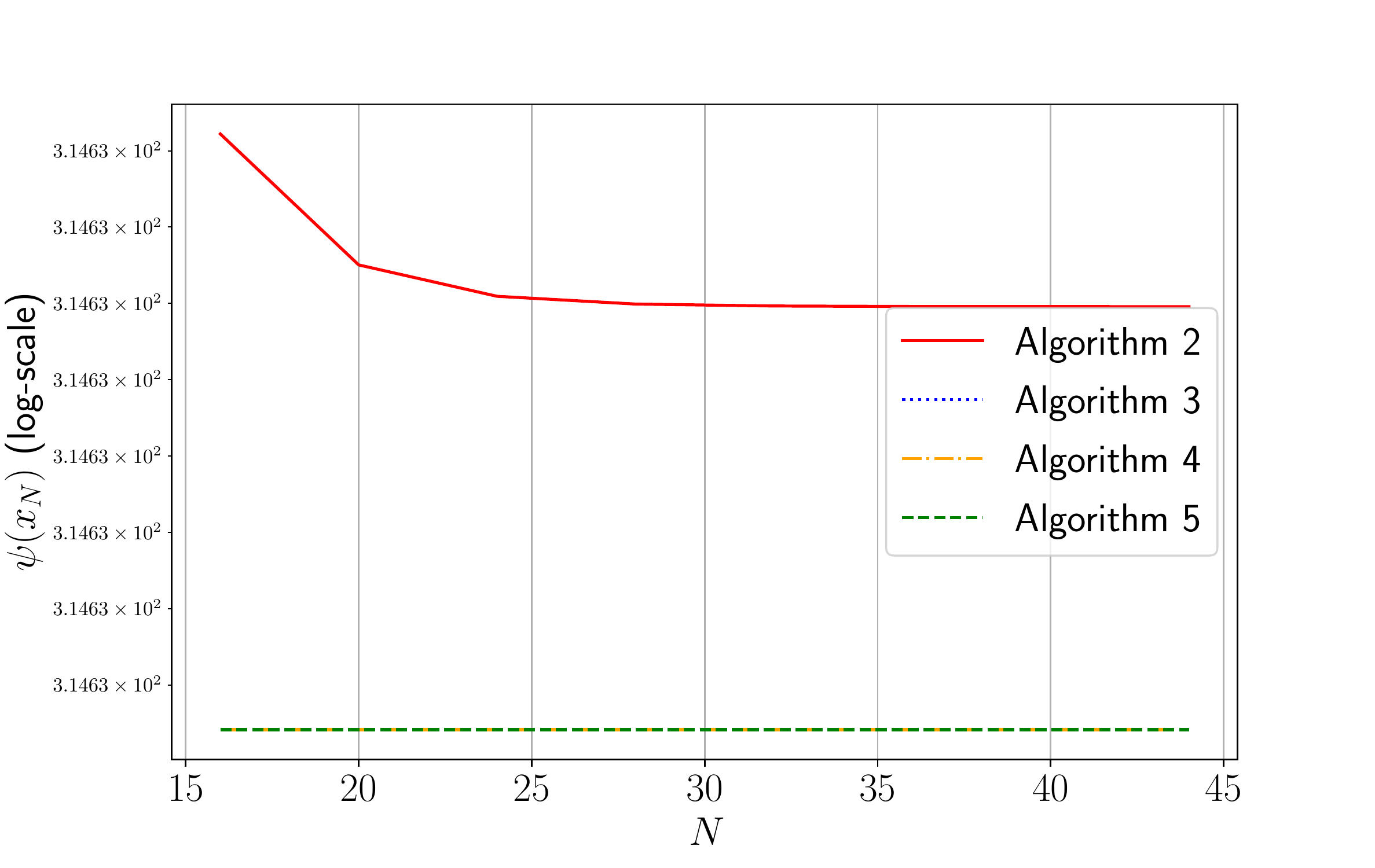}
	\end{minipage}
	\begin{minipage}{0.32\textwidth}
		\centering
		\includegraphics[width=1\textwidth]{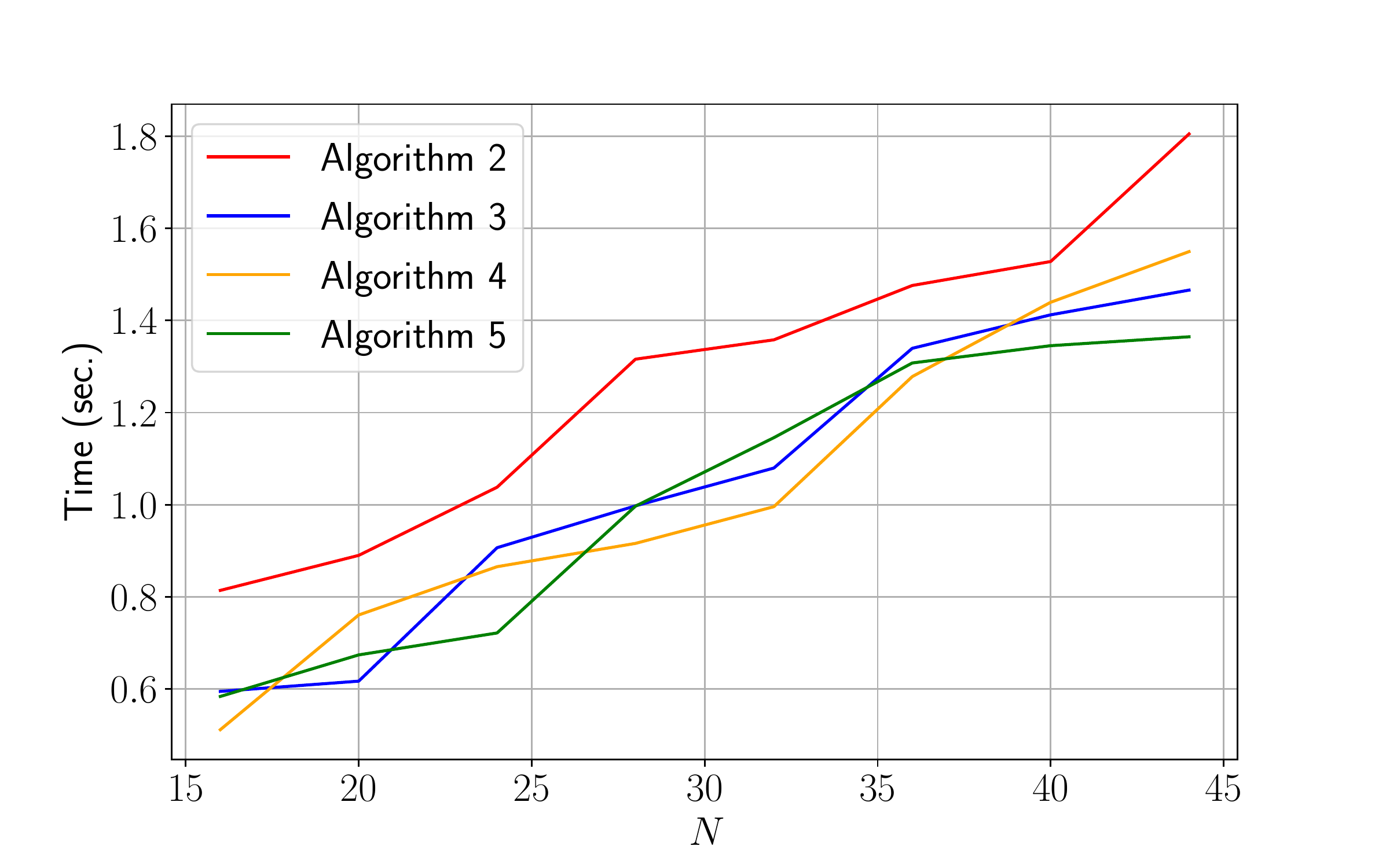}
	\end{minipage}
	\begin{minipage}{0.32\textwidth}
		\centering
		\includegraphics[width=1\textwidth]{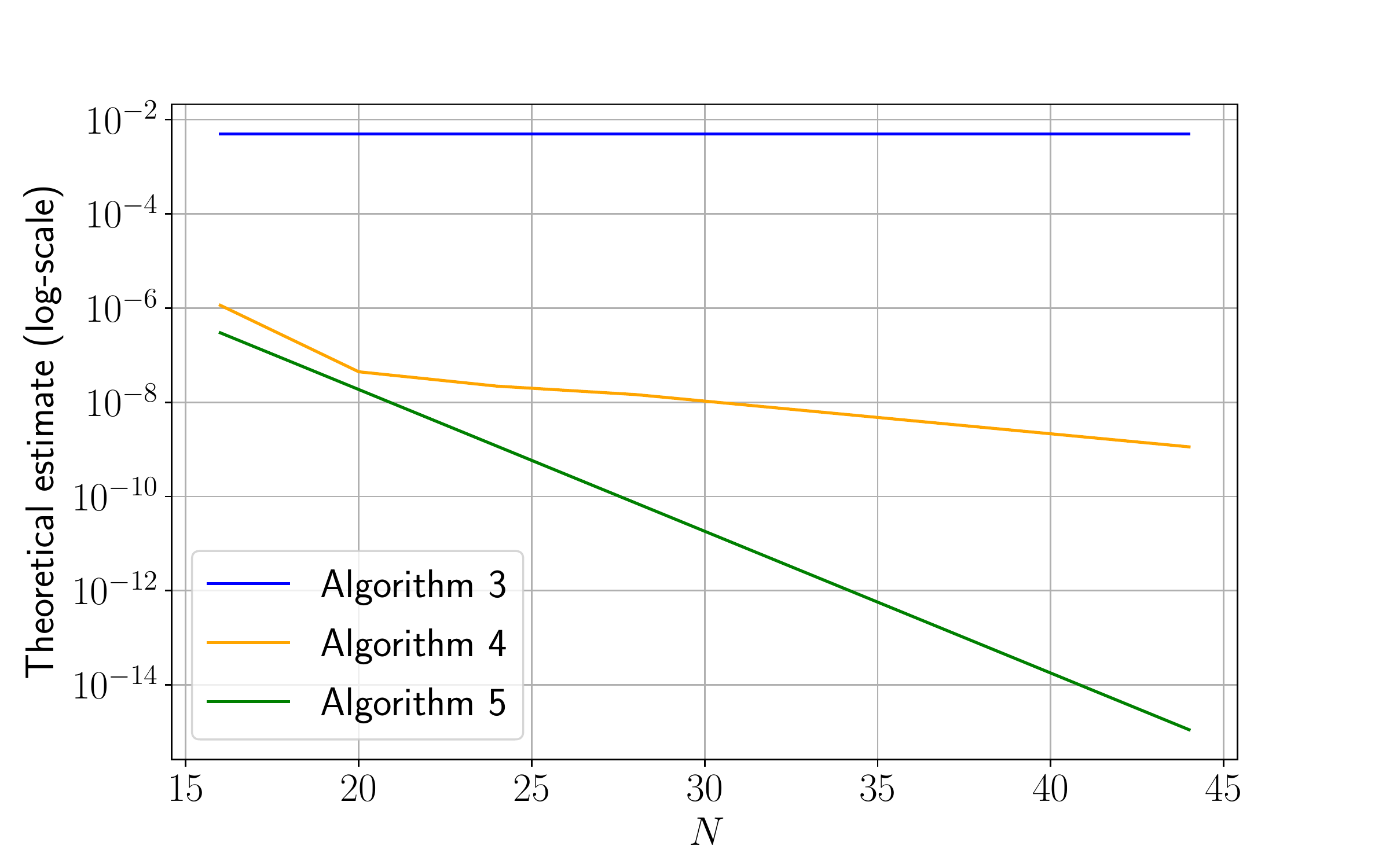}
	\end{minipage}
	\caption{\footnotesize  The results of Algorithms \ref{Alg:RUMP}, \ref{alg2}, \ref{alg1} and \ref{alg3}, for case 2.}
	\label{fig2}
\end{figure*}

From Fig. \ref{fig1}, we can see that the proposed Algorithm \ref{alg3} provides the best quality solution with respect to the value of the objective function {\color{black} $\psi$ at the output point $x_N$}, although {\color{black} the theoretical estimate of this quality} is not very high{\color{black}. Therefore} the efficiency of Algorithm \ref{alg3} is obvious when we look at the value of the objective function at the output point of compared algorithms.  

From Fig. \ref{fig2}, we can see that Algorithm \ref{alg3} also gives the best estimate of the quality of the solution. Also, Algorithms \ref{alg2}, \ref{alg1} and \ref{alg3} give the same objective values at the output points, with approximately the same running time. In this case, we see that Algorithm \ref{alg1} works better than Algorithm \ref{alg2} (it gives better estimate of the quality of the solution), not as in case 1. 

\subsection{One Minty Variational Inequality}

In this subsection we consider variational inequality with the following Lipschitz-continuous and strongly monotone operators (See \cite{arxiv_version}, for more details about the setting of this problem and the setting of its connected conducted experiments with it.).

The first operator is (see Example 5.2 in \cite{khanh2014modified})
    \begin{equation}\label{operator_numerical}
    g: X \subset \mathbb{R}^n \to \mathbb{R}^n, \quad g(x) = x.
    \end{equation}
    
The second operator is
\begin{equation}\label{second_operator}
    g: X \subset \mathbb{R}^n \to \mathbb{R}^n, \quad g(x_1, \ldots, x_n) = \left(x_1, 2^2x_2, 3^2 x_3, \ldots, n^2 x_n\right).
\end{equation}
This operator is $L$-Lipschitz continuous with $L = n^2$ and $\mu$-strongly monotone with $\mu =1$. The condition number for this operator is $\kappa = L/\mu {\color{black} = n^2}$, therefore this operator will be ill-conditioned when $n$ is relatively big. 

For the experiments with operator \eqref{operator_numerical}, the results are presented in Fig. \ref{fig_VI_norm_xout} and \ref{fig_VI_time}, which illustrate the norm $\|x_{\text{out}} - x_*\|_2$, and the running time in seconds as a function of iterations, where $x_{\text{out}}$ is the output of each algorithm. 

In Fig. \ref{fig_VI_norm_xout}, we can not see the graphics of the Algorithms \ref{alg2} and \ref{alg3}, that indicate the distance $\|x_{\text{out}} - x_*\|_2$, because by these algorithms this distance is equal to zero for all considered number of iterations. 

\begin{figure*}[t]
	\begin{minipage}{0.32\textwidth}
		\centering
		\includegraphics[width=1\textwidth]{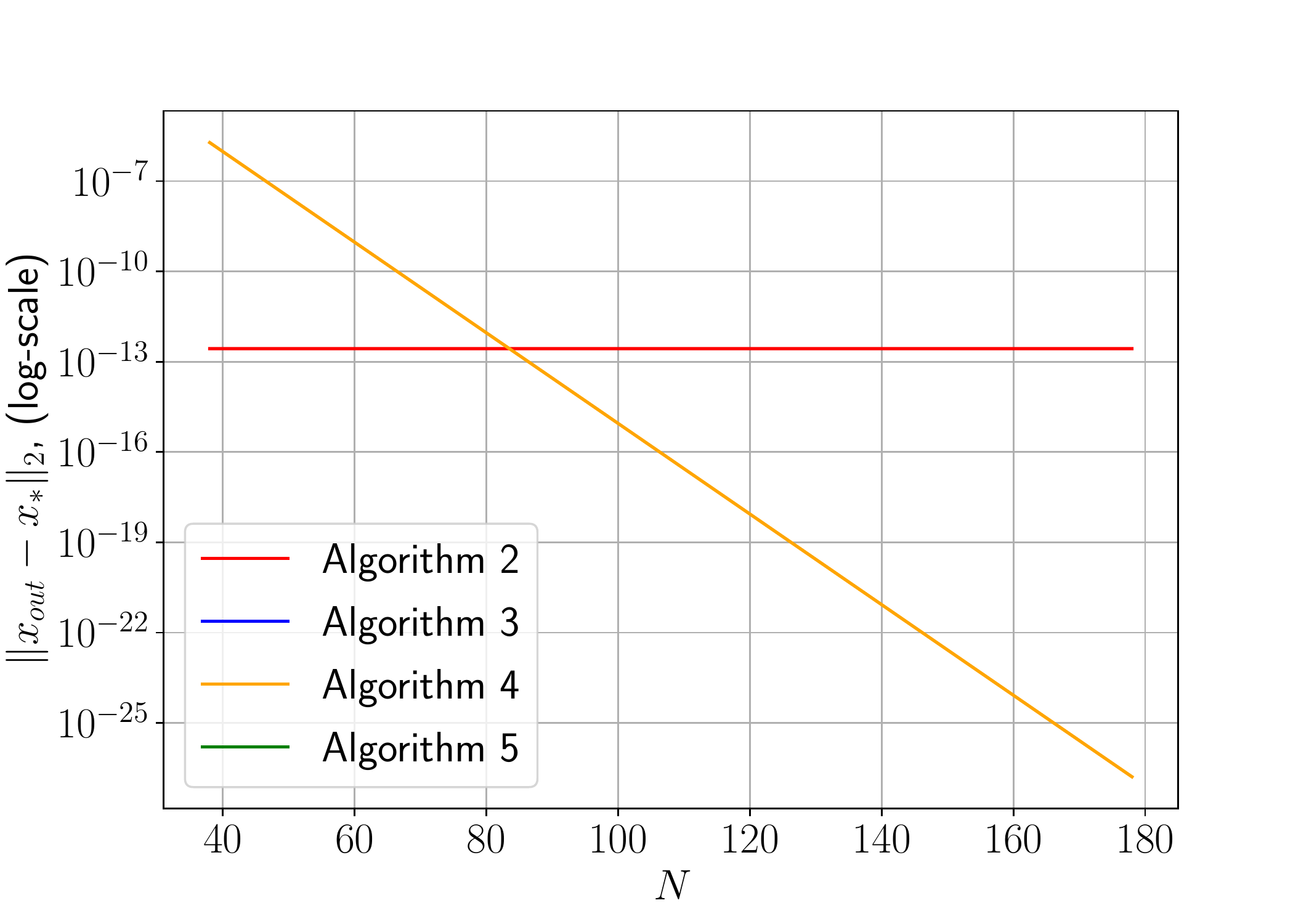}
		\vspace{-0.6cm}%
		{\footnotesize  $r = 1.$}
		\vspace{0.3cm}%
	\end{minipage}
	\begin{minipage}{0.32\textwidth}
		\centering
		\includegraphics[width=1\textwidth]{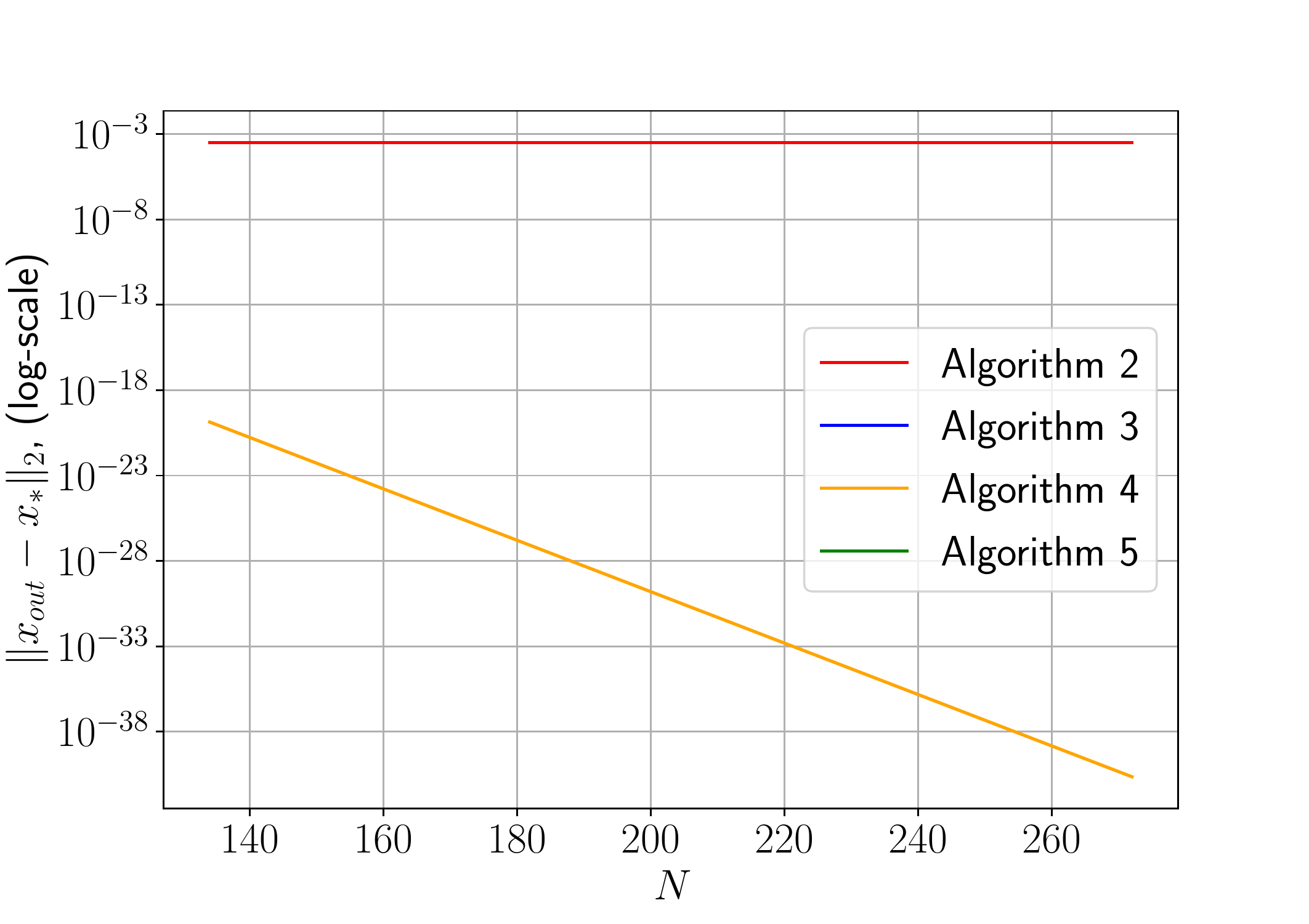}
		\vspace{-0.6cm}%
		{\footnotesize $r=2.$}
		\vspace{0.3cm}%
	\end{minipage}
	\begin{minipage}{0.32\textwidth}
		\centering
		\includegraphics[width=1\textwidth]{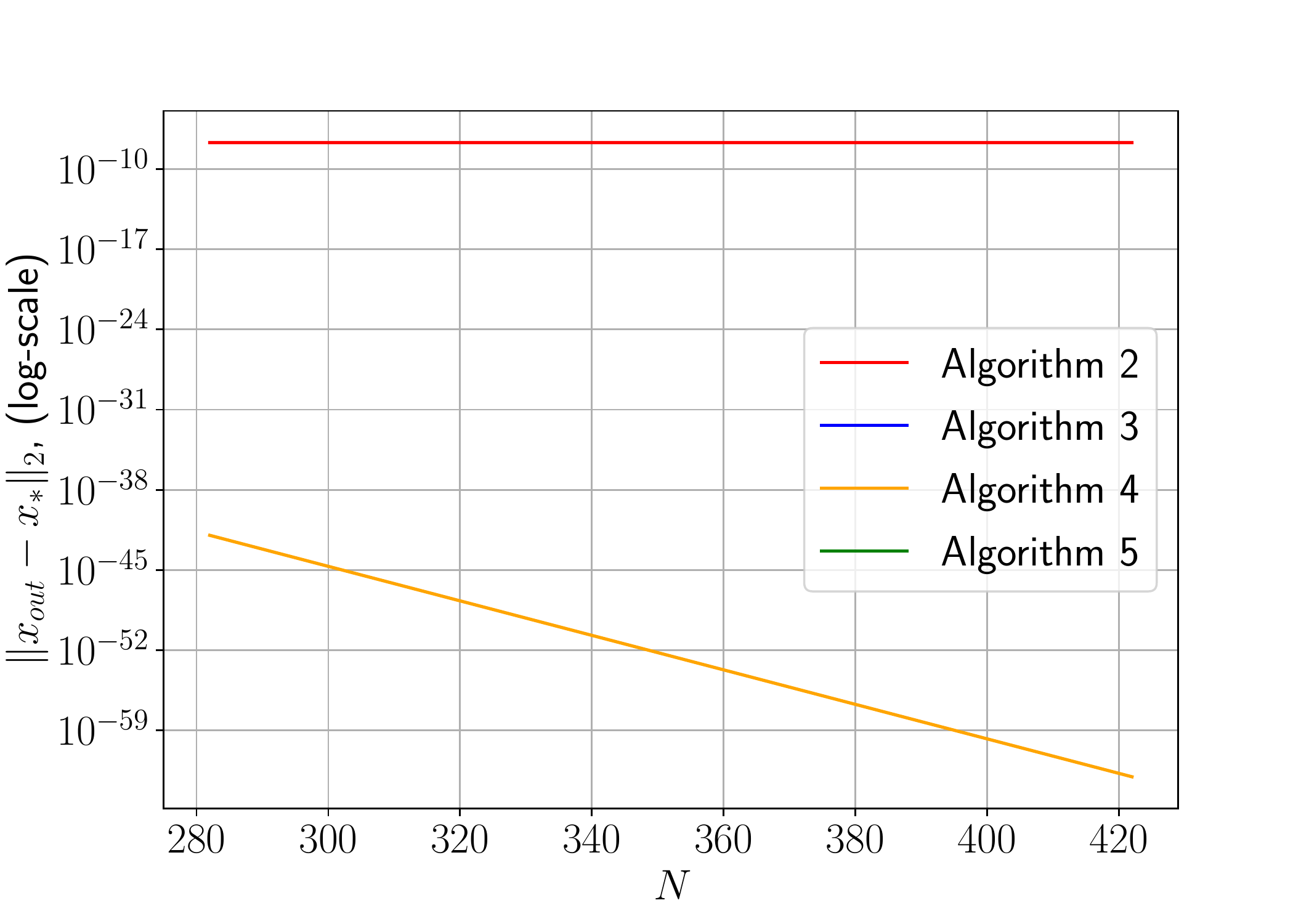}
		\vspace{-0.6cm}%
		{\footnotesize $r=3.$}
		\vspace{0.3cm}%
	\end{minipage}
	\caption{\footnotesize  The results of Algorithms \ref{Alg:RUMP}, \ref{alg1}, \ref{alg2} and \ref{alg3}, for operator \eqref{operator_numerical}, in the set $X = \{x \in \mathbb{R}^n, \|x\|_2 \leqslant r\}$  with different radii $r$ and $ n = 10^6, \varepsilon \in \{10^{-3i}, i = 1,2,\ldots, 8 \}, \delta = 0.01$. }
	\label{fig_VI_norm_xout}
\end{figure*}

\begin{figure*}[t]
	\begin{minipage}{0.32\textwidth}
		\centering
		\includegraphics[width=1\textwidth]{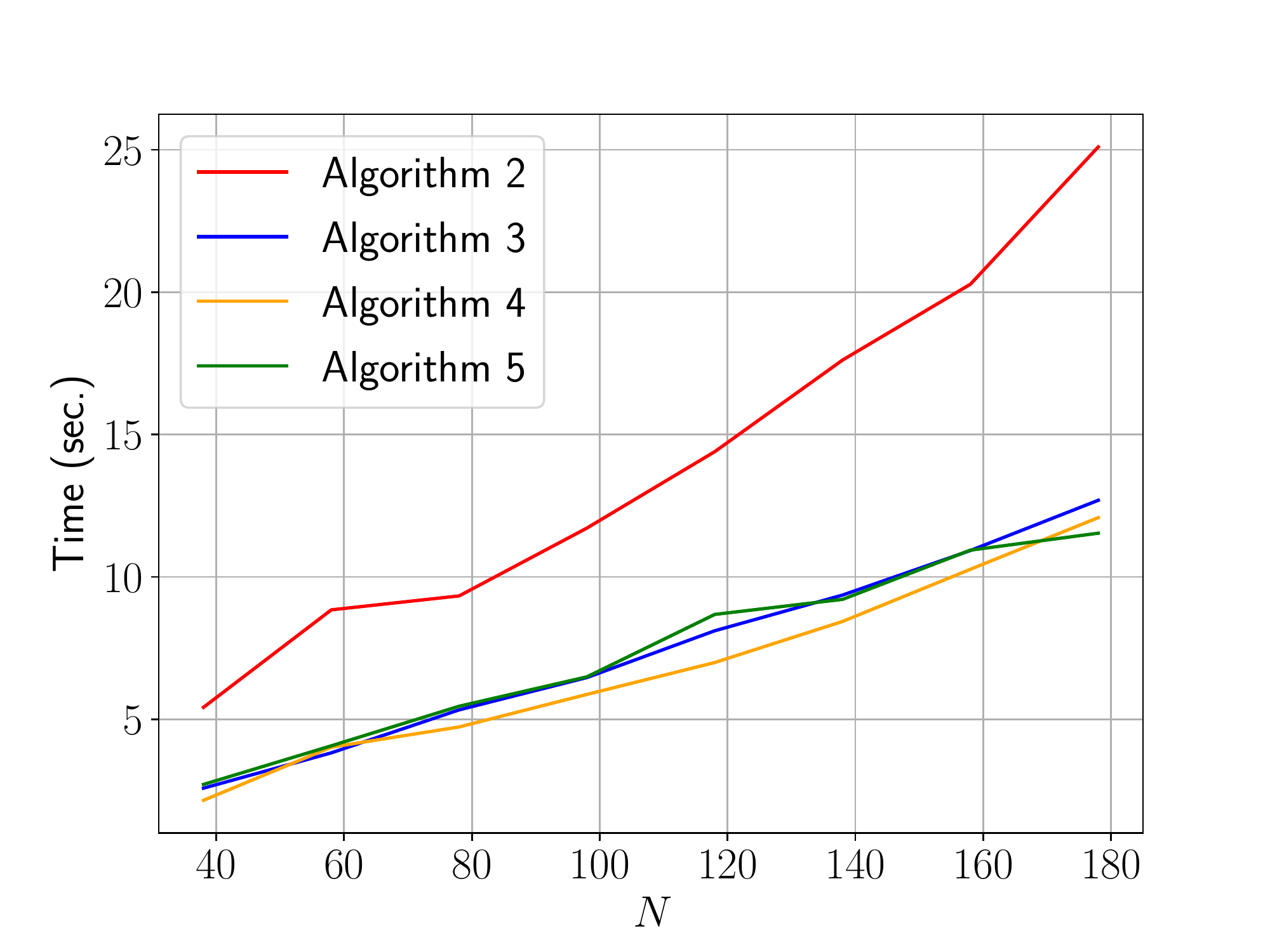}
		\vspace{-0.6cm}%
		{\footnotesize  $r = 1.$}
		\vspace{0.3cm}%
	\end{minipage}
	\begin{minipage}{0.32\textwidth}
		\centering
		\includegraphics[width=1\textwidth]{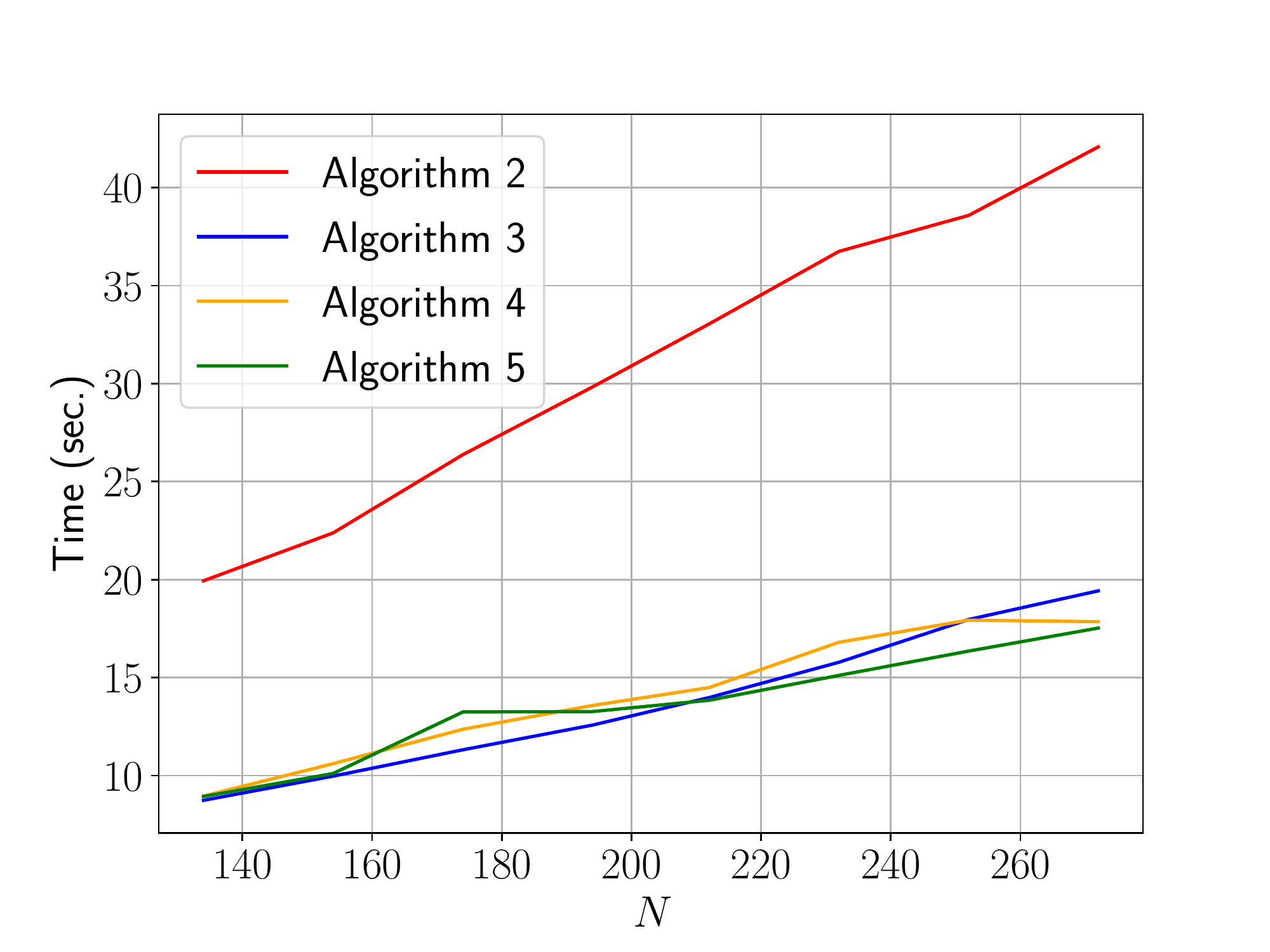}
		\vspace{-0.6cm}%
		{\footnotesize $r=2.$}
		\vspace{0.3cm}%
	\end{minipage}
	\begin{minipage}{0.32\textwidth}
		\centering
		\includegraphics[width=1\textwidth]{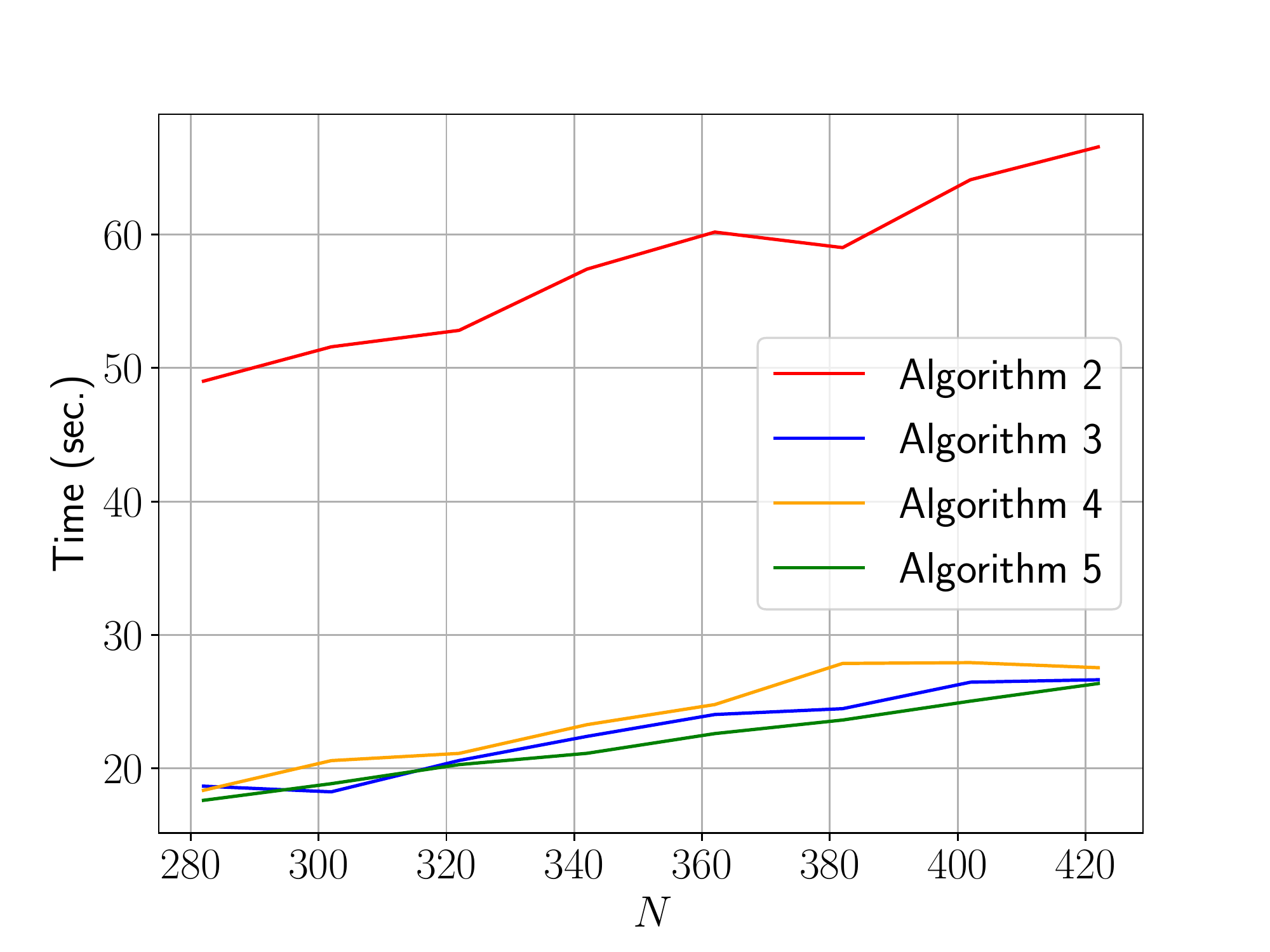}
		\vspace{-0.6cm}%
		{\footnotesize $r=3.$}
		\vspace{0.3cm}%
	\end{minipage}
	\caption{\footnotesize  The results of Algorithms \ref{Alg:RUMP}, \ref{alg1}, \ref{alg2} and \ref{alg3}, for operator \eqref{operator_numerical}, in the set $X = \{x \in \mathbb{R}^n, \|x\|_2 \leqslant r\}$  with different radii $r$ and $ n = 10^6, \varepsilon \in \{10^{-3i}, i = 1,2,\ldots, 8 \}, \delta = 0.01$. }
	\label{fig_VI_time}
\end{figure*}

From the conducted experiments (for operator \eqref{operator_numerical}), we can see that the shape of the feasible set very much affects the progress of the work of the proposed algorithms. We note that when we increase $r$ (the radius of the ball $X$), the corresponding running time of the compared algorithms is also increased. Also, from Fig. \ref{fig_VI_norm_xout} and \ref{fig_VI_time}, we can see that the proposed Algorithms \ref{alg2} and \ref{alg3}, for any value of the radius $r$,  are the best, where they give the solution of the problem under consideration with very high quality and at the same (approximately) running time. Also, we can see that Algorithm \ref{alg1} always works better than Algorithm \ref{Alg:RUMP}.  


Now, for the experiments with operator \eqref{second_operator}, 
the results are presented in Fig. \ref{res_second_operator}, which illustrates the norm $\|x_{\text{out}} - x_*\|_2$, the running time in seconds as a function of iterations, and the theoretical estimates \eqref{estimate_UMP}, \eqref{estim_alg4} and \eqref{estim_alg5} of the quality of the solution, for Algorithms \ref{Alg:UMP}, \ref{alg2} and \ref{alg3}, respectively.

\begin{figure*}[htp]
	\centering
	\begin{minipage}{0.32\textwidth}
		\centering
		\includegraphics[width=1\textwidth]{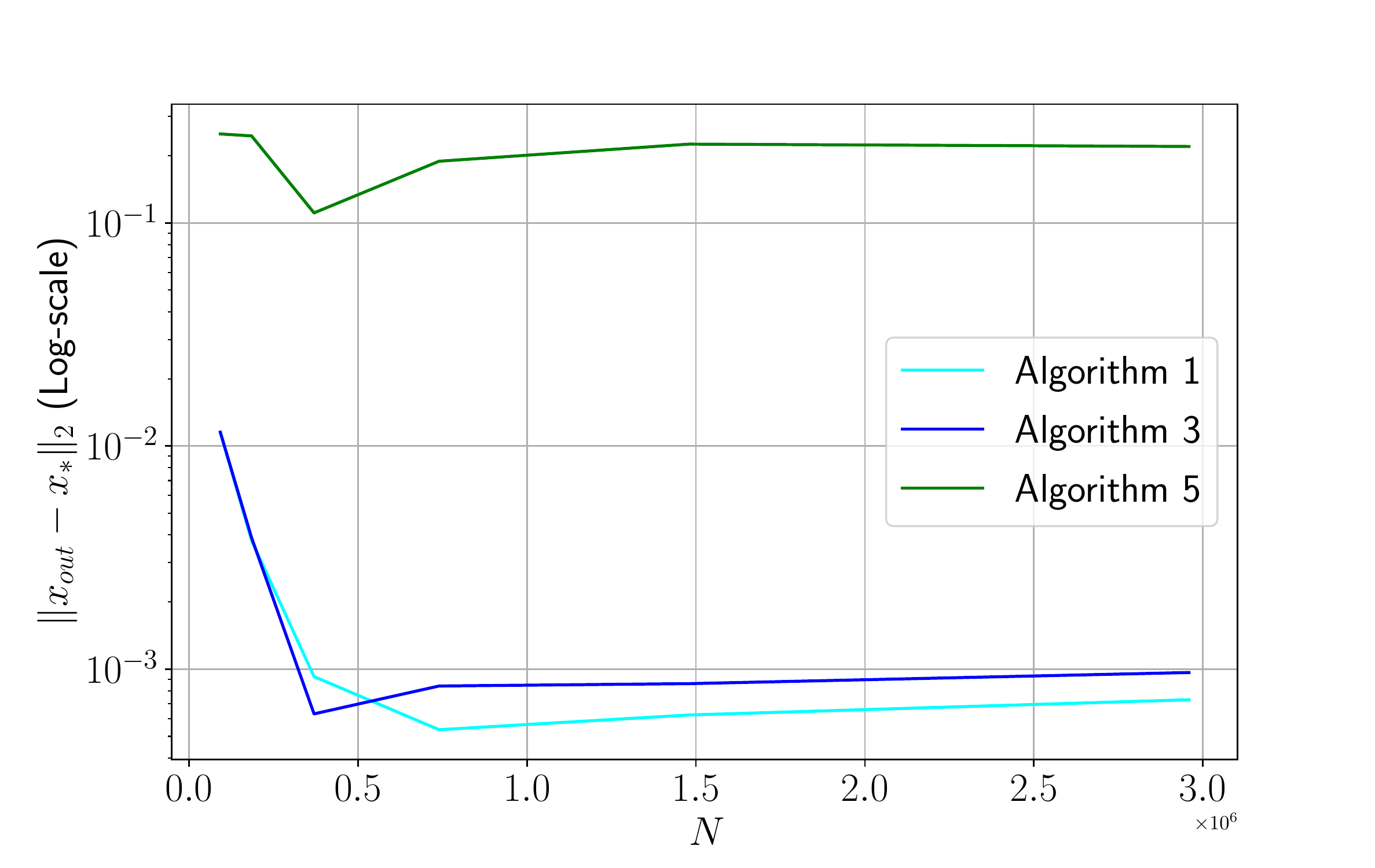}
		\vspace{-0.6cm}%
		\vspace{0.3cm}%
	\end{minipage}
	\centering
	\begin{minipage}{0.32\textwidth}
		\centering
		\includegraphics[width=1\textwidth]{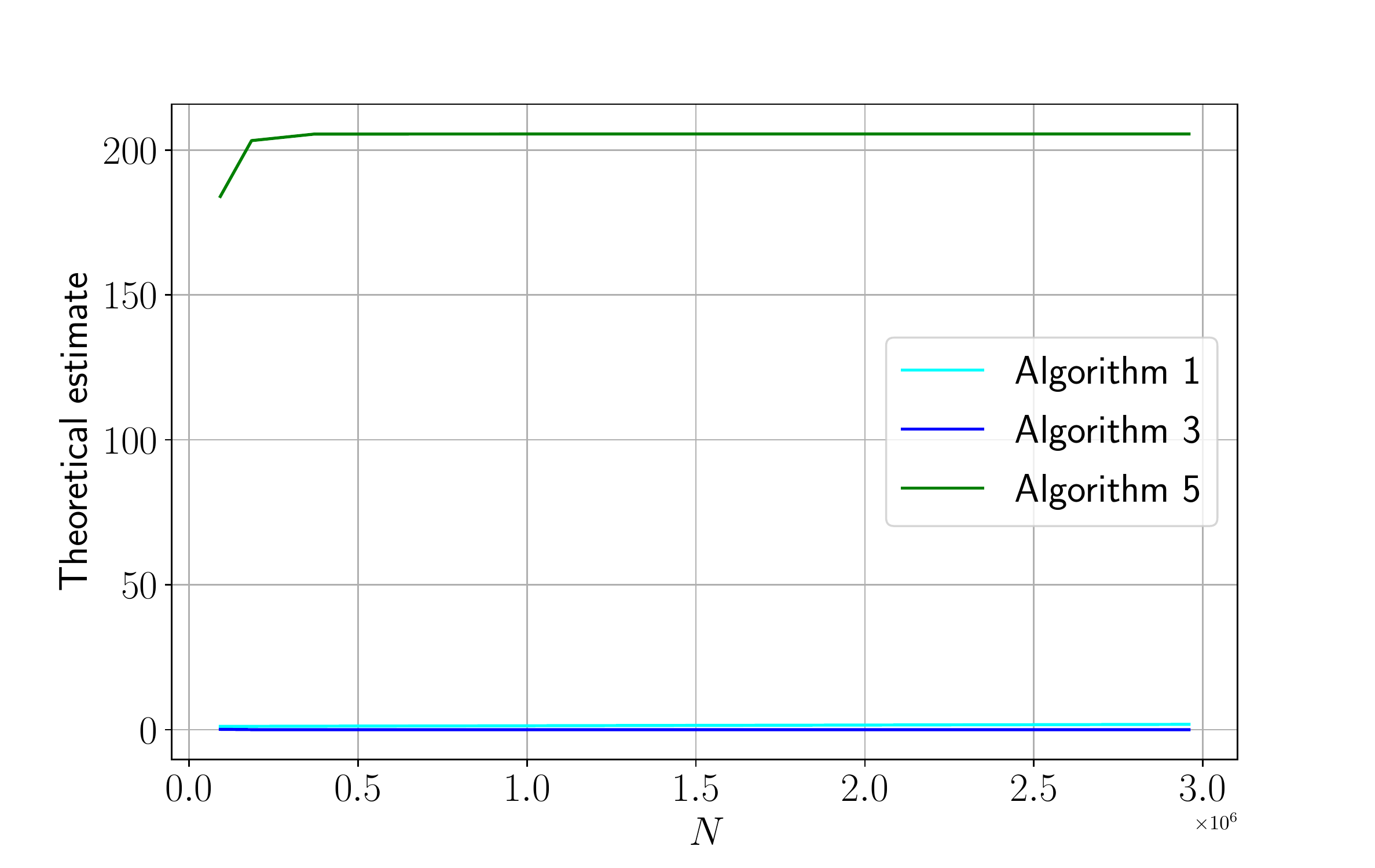}
		\vspace{-0.6cm}%
		\vspace{0.3cm}%
	\end{minipage}
	\centering
	\begin{minipage}{0.32\textwidth}
		\centering
		\includegraphics[width=1\textwidth]{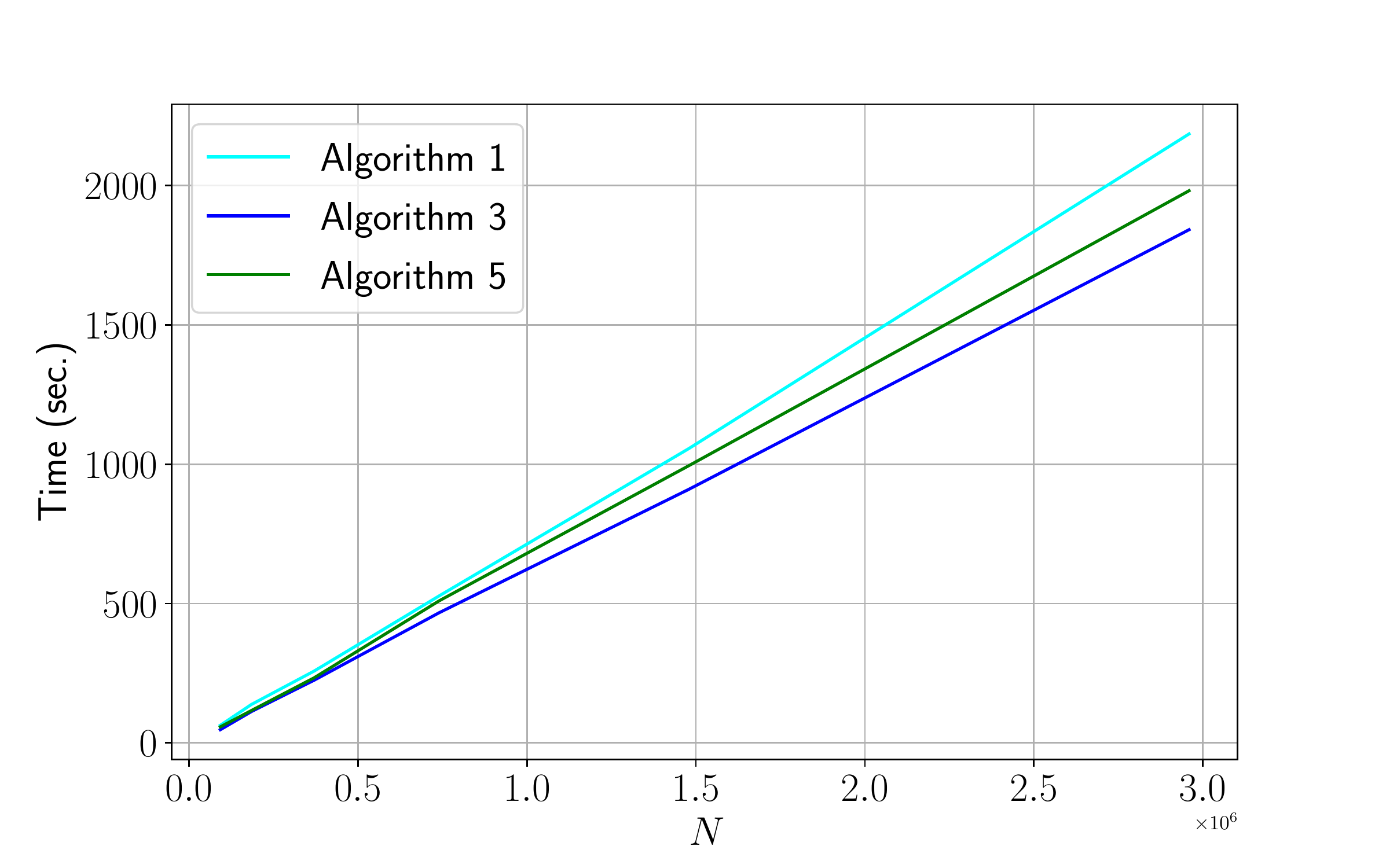}
		\vspace{-0.6cm}%
		\vspace{0.3cm}%
	\end{minipage}
	\caption{\footnotesize  The results of Algorithms \ref{Alg:UMP}, \ref{alg2}  and \ref{alg3}, for operator \eqref{second_operator}. }
	\label{res_second_operator}
\end{figure*}

From Fig. \ref{res_second_operator}, we can see that Algorithm \ref{alg3} is the worst. Algorithm \ref{Alg:UMP}, with respect to the distance $\|x_{\text{out}} - x_*\|_2$, gives results better than Algorithms \ref{alg2}. Also, the theoretical estimate of the quality of the solution by Algorithm \ref{Alg:UMP} is approximately the same as by Algorithm \ref{alg2}. The difference between the running times of Algorithms \ref{Alg:UMP} and \ref{alg2} is not big. Note that, Algorithm \ref{Alg:UMP} is applicable to a wider type of problem and can work better than Algorithms \ref{alg2}, \ref{alg1}  and \ref{alg3} with small $\mu$. At the same time for variational inequalities with strongly monotone operators, it can be restarted in such a way that Algorithm \ref{Alg:RUMP} will have similar convergence rate estimates. Thus, Algorithm \ref{Alg:RUMP} will be the best for the problems with an ill-conditioned operator.

\section*{Conclusion}

In this paper, we study adaptive first-order methods for variational inequalities from the class of relatively strongly monotone operators recently introduced in \cite{Stonyakin}. Our research is motivated, in particular, by the recently proposed technique \cite{Sid,Sidford} for strongly convex-concave saddle point problems, which allows one to obtain the complexity estimates of the accelerated methods.
First of all, the paper deals with the issue of adaptive tuning of the method to the global smoothness parameters of the saddle point problem. Moreover, an essential feature is the consideration of operators with a generalized condition of relative $L$-Lipschitz property and the corresponding generalizations of smoothness for the class of saddle point problems under consideration.
Based on the methods from \cite{Sid,Sidford}, 
we proposed algorithms for solving variational inequalities with relatively strongly monotone operators and obtained estimates of their convergence 
We also presented some numerical experiments, which demonstrate the effectiveness of the proposed methods. We considered an example of the convex optimization problem with functional constraints, and an example of the Minty variational inequality. The conducted experiments showed that the proposed Algorithms \ref{alg2}, \ref{alg1} and \ref{alg3} without using the technique of restarts work better than algorithm with restarts (Algorithm \ref{Alg:RUMP}) and vice versa for an example of the variational inequality with an ill-conditioned operator.

\end{document}